\documentclass[12pt,english]{amsart}
\usepackage[latin9]{inputenc}
\usepackage{geometry}
\usepackage{soul}
\usepackage{color}
\usepackage{amstext}
\usepackage{amsthm}
\usepackage{amssymb}
\usepackage{esint}

\usepackage[colorlinks=true, pdfstartview=FitV, linkcolor=blue, citecolor=blue, urlcolor=blue]{hyperref}
\usepackage{doi}
\usepackage{url}

\makeatletter
\numberwithin{equation}{section}
\numberwithin{figure}{section}
  \theoremstyle{plain}
  \newtheorem*{thm*}{\protect\theoremname}
\theoremstyle{plain}
\newtheorem{thm}{\protect\theoremname}

\theoremstyle{plain}
\newtheorem{thmInt}{\protect\theoremname}

  \theoremstyle{plain}
  \newtheorem{cor}{\protect\corollaryname}
  \theoremstyle{definition}
  \newtheorem{defn}{\protect\definitionname}
  \theoremstyle{plain}
  \newtheorem{lem}{\protect\lemmaname}
  \theoremstyle{plain}
  
  \theoremstyle{plain}
  \newtheorem*{lem*}{\protect\lemmaname}

\usepackage{amsmath}
\DeclareMathOperator{\supp}{supp}

\makeatother

\usepackage{babel}
  \providecommand{\definitionname}{Definition}
  \providecommand{\lemmaname}{Lemma}
  \providecommand{\propositionname}{Proposition}
  \providecommand{\theoremname}{Theorem}
\providecommand{\corollaryname}{Corollary}
\providecommand{\theoremname}{Theorem}

\def\XXint#1#2#3{{\setbox0=\hbox{$#1{#2#3}{\int}$}
    \vcenter{\hbox{$#2#3$}}\kern-.5\wd0}}

\begin{document}

\title{A quantitative approach to weighted Carleson Condition}

\author{I.P. Rivera-R\'{\i}os}
\address{Israel P. Rivera-R\'{\i}os, IMUS \& Departamento de Análisis Matemático, Universidad de Sevilla, Sevilla, Spain}

\begin{abstract}Quantitative versions of weighted estimates obtained by F. Ruiz and J.L. Torrea \cite{R,RT} for the operator
\[
\mathcal{M}f(x,t)=\sup_{x\in Q,\,l(Q)\geq t}\frac{1}{|Q|}\int_{Q}|f(x)|dx \qquad x\in\mathbb{R}^{n}, \, t \geq0
\]
are obtained.  As a consequence, some sufficient conditions for the boundedness
of $\mathcal{M}$ in the two weight setting in the spirit of the results
obtained by C. P\'erez and E. Rela \cite{PR} and very recently by M.T.
Lacey and S. Spencer \cite{LS} for the Hardy-Littlewood maximal operator are derived. As a byproduct some new quantitative
estimates for the Poisson integral are obtained.
\end{abstract}

\maketitle

\section{Introduction}

In 1962 L. Carleson \cite{C} studied a variation of the Hardy-Littlewood
maximal operator defined on $\mathbb{R}_{+}^{n+1}=\{(x,t)\,:\,x\in\mathbb{R}^{n},t\geq0\}$
as follows
\[
\mathcal{M}f(x,t)=\sup_{x\in Q,\,l(Q)\geq t}\frac{1}{|Q|}\int_{Q}|f(x)|dx.
\]
The interest of the study of this operator stems from the fact that
it controls pointwise the Poisson integral. Indeed, if we call
\[
P(x,t)=c_{n}\frac{t}{\left(|x|^{2}+t^{2}\right)^{\frac{n+1}{2}}}
\]
the Poisson integral is defined as
\[
Pf(x,t)=\int_{\mathbb{R}^{n}}f(y)P(x-y,t)dy\quad x\in\mathbb{R}^{n},\,t\geq0
\]
and it's not hard to check that (cf. \cite{GCRdF} Chapter 2) 
\begin{equation}Pf(x,t)\leq c_{n}\mathcal{M}f(x,t).
\label{PoissonPuntual}\end{equation}
Carleson characterized the positive Borel
measures on $\mathbb{R}_{+}^{n+1}$ such that $\mathcal{M}$ is of
strong type $(p,p)$ for $p>1$ and of weak type $(1,1)$. The precise statement
of his result is the following
\begin{thmInt}
Let $\mu$ be a positive Borel measure  on $\mathbb{R}_{+}^{n+1}$ and let $1<p<\infty$. Then, the
following statements are equivalent,
\begin{enumerate}
\item $\mathcal{M}:L^{p}(\mathbb{R}^{n})\rightarrow L^{p}(\mathbb{R}_{+}^{n+1},\mu)$\,  or
\,\textup{$\mathcal{M}:L^{1}(\mathbb{R}^{n})\rightarrow L^{1,\infty}(\mathbb{R}_{+}^{n+1},\mu)$}
\item There is a constant $C$ such that for each cube $Q$ in $\mathbb{R}^{n}$
\[
\mu(\tilde{Q})\leq C|Q|
\]
where $\tilde{Q}=\left\{ (x,t)\,:\,x\in Q,\,0\leq t\leq l(Q)\right\} .$
\end{enumerate}
\end{thmInt}
In fact, it can be established that exactly the same result holds
for the Poisson integral (cf. \cite{GCRdF} Chapter 2). Condition (2) in the preceding  Theorem
is the so called ``Carleson Condition''.  Later on, in the 70's \cite{FS}
Fefferman and Stein obtained the following weighted result

\begin{thmInt}Let $\mu$ be a positive Borel measure  on $\mathbb{R}_{+}^{n+1}$ and let $w$ be a weight in $\mathbb{R}^{n}$ such that
\[
M\mu(x)=\sup_{x\in Q} \frac{\mu(\tilde{Q})}{|Q|}\leq Cw(x).
\]
Then we have both  \, $\mathcal{M}:L^{p}(\mathbb{R}^{n},w)\rightarrow L^{p}(\mathbb{R}^{n+1},\mu)$,
\textup{and}  $\mathcal{M}:L^{1}(\mathbb{R}^{n},w)\rightarrow L^{1,\infty}(\mathbb{R}_{+}^{n+1},\mu)$
\end{thmInt}
This result was extended in the 80's by F. Ruiz and J. L. Torrea \cite{R,RT}.
The conditions obtained by these authors are in the spirit of the ones
obtained by B. Muckenhoupt's \cite{M} and E. Sawyer's \cite{S} for
the Hardy-Littlewood maximal function. The main results are the following.
\begin{thmInt}\label{Thm:R}
Let $\mu$ be a positive Borel measure  on $\mathbb{R}_{+}^{n+1}$ and let $w$ be a weight in $\mathbb{R}^{n}$. Let also $1< p<\infty$,  then the following statements are equivalent: 
\begin{enumerate}
\item  $\mathcal{M}:L^{p}(\mathbb{R}^{n},w)\longrightarrow L^{p,\infty}(\mathbb{R}_{+}^{n+1},\mu)$
\item  $\sup_{Q}\frac{\mu(\tilde{Q})}{|Q|}\left(\frac{1}{|Q|}\int_{Q}w(x)^{1-p'} \,dx\right)^{p-1}<\infty$.
\end{enumerate}
\end{thmInt}

\begin{thmInt}\label{Thm:RT}Let $\mu$ be a positive Borel measure  on $\mathbb{R}_{+}^{n+1}$ and let $v$ be a weight in $\mathbb{R}^{n}$. Let also $1<p<\infty$, then  the following
conditions are equivalent:
\begin{enumerate}
\item  $\mathcal{M}:L^{p}(\mathbb{R}^{n},{v})\longrightarrow L^{p}(\mathbb{R}_{+}^{n+1},\mu)$
\item There is a constant $c>0$ such that for each cube $Q$
\[
\int_{\tilde{Q}}\mathcal{M}(v^{1-p'}\chi_{Q})(x,t)^{p}d\mu(x,t)\leq c\int_{Q}v^{1-p'}(x)dx<\infty
\]
\end{enumerate}
\end{thmInt}

Theorems \ref{Thm:R} and \ref{Thm:RT} characterize qualitatively the weak and the strong type $L^p$ boundedness of $\mathcal{M}$ in the sense that they characterize the $L^p$ boundedness but they don't provide a quantitative relationship between the operator norm and the relevant constants associated to the 
couple $(\mu,v)$.   Quantitative estimates for the main operators in Harmonic Analysis have been widely considered by many authors in the last years. 
The first result can be traced back to the work of Buckley \cite{B} where a quantitative version of the classical Muckenhupt theorem is obtained, namely 
\[\|Mf\|_{L^p(w)}\leq c_{n,p}[w]_{A_p}^\frac{1}{p-1}\|f\|_{L^p(w)},\] 
where $M$ is the classical Hardy-Littlewood maximal function.  Later on and motivated by the $A_2$  conjecture for the Ahlfors-Beurling transform formulated in \cite{AIS}, the sharp dependence of weighted inequalities on the $A_p$ constant has been studied thoroughly for operators such as Calder\'on-Zygmund operators (see for instance \cite{PV,P1,P2,H,Le2} or the more recent works \cite{CAR,LN,La,Le3}), rough singular integrals (c.f. \cite{HRT}), commutators (c.f. \cite{CPP}) or the square function (for example \cite{HL}).

Our aim in this paper is to obtain quantitative versions of Theorems \ref{Thm:R} and \ref{Thm:RT} which, as a consequence,  will provide two new quantitative sufficient conditions for the boundedness of $\mathcal{M}$. Since $\mathcal{M}$ controls pointwise $P$, as a direct consequence of those quantitative estimates for $\mathcal{M}$ we will also obtain corresponding quantitative estimates for the Poisson integral.

This paper is organized as follows. In section 2 we state our main
results. In section 3 we introduce some definitions needed to understand
in full detail the main results. Finally in section 4 we prove our
main results.

\subsection*{Aknowledgements }

I would like to thank my advisor Carlos P\'erez for turning my attention
to this problem and for his guidance and support during the elaboration
of this paper.

\section{Main Results}
Before we state our main results we would like to note that the precise definitions of the operators and the constants involved are summarized in section \ref{sec:Preliminaries-and-definitions}.

Our first of result is a quantitative characterization of the weak-type $(p,p)$ boundedness.

\begin{thm}\label{Th:Debilpp} 

Let $\mu$ be a positive Borel measure  on $\mathbb{R}_{+}^{n+1}$ and let $\sigma$ be a weight in $\mathbb{R}^{n}$. Let also $1<p<\infty$, then
\[ 
\|\mathcal{M}(\sigma\cdot) \|_{L^p(\mathbb{R}^n,\sigma)\rightarrow L^{p,\infty}(\mathbb{R}_+^{n+1},\mu)}\simeq  [\mu,\sigma]_{A'_{p}}^{\frac{1}{p}}\]
\end{thm}

\begin{thm}
\label{Thm:CarlesonMax}
Let $\mu$, $\sigma$ and $p$ as in the previous theorem. Then
\[
\|\mathcal{M}(\sigma\cdot)\|_{L^{p}(\mathbb{R}^{n},\sigma)\rightarrow L^{p}(\mathbb{R}_{+}^{n+1},\mu)} \simeq [\mu,\sigma]_{S'_{p}}. 
\]
\end{thm}

Relying on the preceding result and following the proof for the Hardy-Littlewood
maximal operator obtained by P\'erez and Rela \cite{PR} we can obtain
the following quantitative sufficient condition for the boundedness
of the operator in the weighted setting.
\begin{thm}
\label{Thm:QuantPR}Let $\mu$, $\sigma$ and $p$ as before. Also, let $\Phi$ be a Young function
with conjugate function $\bar{\Phi}$. Then 
\[
\|\mathcal{M}(\cdot\sigma)\|_{L^{p}(\sigma)\rightarrow L^{p}(\mathbb{R}_{+}^{n+1},\mu)}\lesssim\left([\mu,\sigma,\Phi]_{A'_{p}}[\sigma,\bar{\Phi}]_{W_{p}}\right)^{\frac{1}{p}}
\]
\end{thm}
The definitions of the constants involved in this result can be found
in section \ref{sec:Preliminaries-and-definitions}.  These kind of conditions, often called ``bump" conditions, were introduced in the 90's in \cite{Pe1} and considered to study sharp two weighted estimates for Singular Integrals \cite{Pe2} and have been often used in recent literature. 

A very interesting recent result obtained by L. Slav\'{\i}kov\'a is that this condition is not neccesary. We remit to \cite{Sl} for details.

The following result generalizes the result due to Lacey and Spencer \cite{LS} which is based on the very recent idea of entropy introduced in \cite{TV}.

\begin{thm}
\label{Thm:QuantLS}Let $\mu$, $\sigma$ and $p$ as before. Then
\[
\|\mathcal{M}(\cdot\sigma)\|_{L^{p}(\sigma)\rightarrow L^{p}(\mathbb{R}_{+}^{n+1},\mu)}\lesssim\left\lceil \mu,\sigma\right\rceil _{p,\varepsilon}^{\frac{1}{p}}.
\]
\end{thm}

It's clear that if we replace $\mathcal{M}$ by the standard Hardy-Littlewood
maximal operator and if we take $\mu(x,t)=w(x)\delta_{0}(t)$, where
$\delta_{0}$ denotes Dirac's delta, we recover all the results in
their classical setting.

As a consequence of the preceding results and \eqref{PoissonPuntual}, we obtain also analogue estimates for the Poisson integral. We compile all those estimates in the following

\begin{cor}
Let $\sigma$ be a weight in $\mathbb{R}^{n}$ and let $\mu$ be a Borel measure on $\mathbb{R}_{+}^{n+1}$
and $1<p<\infty$. Then the Poisson integral satisfies the following estimates
\begin{itemize}
\item
 $\|P(\cdot\sigma)\|_{L^{p}(\mathbb{R}^{n},\sigma)\rightarrow L^{p,\infty}(\mathbb{R}_{+}^{n+1},\mu)}\lesssim[\mu,\sigma]_{A'_{p}}^{{\frac{1}{p}}}$
\item $\|P(\cdot\sigma)\|_{L^{p}(\mathbb{R}^{n},\sigma)\rightarrow L^{p}(\mathbb{R}_{+}^{n+1},\mu)}\lesssim[\mu, \sigma]_{S'_{p}}$
\item $\|P(\cdot\sigma)\|_{L^{p}(\sigma)\rightarrow L^{p}(\mathbb{R}_{+}^{n+1},\mu)}\lesssim\left([\mu,\sigma,\Phi]_{A_{p}}[\sigma,\bar{\Phi}]_{W^{p}}\right)^{\frac{1}{p}}$
\item $\|P(\cdot\sigma)\|_{L^{p}(\sigma)\rightarrow L^{p}(\mathbb{R}_{+}^{n+1},\mu)}\lesssim\left\lceil \sigma,\mu\right\rceil _{p,\varepsilon}^{\frac{1}{p}}.$
\end{itemize}
\end{cor}

\section{\label{sec:Preliminaries-and-definitions}Preliminaries and definitions}

\subsection{Basic definitions}
In this section we recall some basic facts that play a main role in this paper. We also give the precise definitions of the operators and the constants used in our results.

\begin{defn}
Let $f$ a locally integrable function in  $\mathbb{R}^{n}$, we define the function
$\mathcal{M}f$ on $\mathbb{R}_{+}^{n+1}=\{(x,t)\,:\,x\in\mathbb{R}^{n},t\geq0\}$
as
\[
\mathcal{M}f(x,t)=\sup_{x\in Q,\,l(Q)\geq t}\frac{1}{|Q|}\int_{Q}|f(y)|dy.
\]

\end{defn}
We also consider  dyadic versions of this operator. First we recall the definition of dyadic grid (c.f. \cite{Le1}).
\begin{defn} We say that a family of cubes $\mathcal{D}$ is a dyadic grid if it satisfies the following properties
\begin{enumerate}
\item For any $Q\in\mathcal{D}$ its sidelength $l_Q$ is of the form $2^k$, $k\in\mathbb{Z}$
\item If $Q_1,Q_2\in\mathcal{D}$ then $Q_1\cap Q_2\in{Q_1,Q_2,\emptyset}$
\item The cubes of a fixed sidelength $2^k$ form a partition of $\mathbb{R}^n$
\end{enumerate} 
\end{defn}

We will also use following
\begin{defn}Let $\mathcal{D}$ a dyadic grid. Given a cube $Q\in\mathcal{D}$ we call $\mathcal{D}(Q)$ the family of all the cubes of $\mathcal{D}$ that are contained in $Q$.
\end{defn}

\begin{defn}Let $f$ a locally integrable function in  $\mathbb{R}^{n}$ and $\mathcal{D}$ a dyadic grid. We define
the function $\mathcal{\mathcal{N}^{D}}f$ on $\mathbb{R}_{+}^{n+1}=\{(x,t)\,:\,x\in\mathbb{R}^{n},t\geq0\}$ as
\[
\mathcal{N^{D}}f(x,t)=\sup_{\stackrel{{\scriptstyle x\in Q\in\mathcal{D}}}{l(Q)\geq t}}\frac{1}{|Q|}\int_{Q}|f(y)|dy.
\]
We shall drop the superscript $\mathcal{D}$ when we work with just one dyadic grid.
\end{defn}
We denote by $\tilde{Q}$ the cube built from a cube $Q$ as follows
\[
\tilde{Q}=\left\{ (x,t)\in\mathbb{R}_{+}^{n+1}\,:\,x\in Q,\text{ and }0\leq t<l(Q)\right\} ,
\]
in other words, $\tilde{Q}$ is the cube in $\mathbb{R}_{+}^{n+1}$
having $Q$ as a face.

Using the argument given in \cite[Lemma 5.38 pg. 111]{CUMP}, we can
obtain the following Lemma that we will use in the sequel.
\begin{lem}
\label{Lemma:CUMP5.38}Given a Young function $\Phi$ for every $Q$
and every $(x,t)\in\tilde{Q}$ we have that
\[
\mathcal{M}(f\chi_{Q})(x,t)=\sup_{\stackrel{{\scriptstyle x\in P\subseteq Q}}{l(P)\geq t}}\frac{1}{|P|}\int_{P}|f(y)|dy.
\]

\end{lem}

\subsection{Orlicz averages}

We recall that $\Phi$ is a Young function if it is a continuous,
nonnegative, strictly increasing and convex function defined on $[0,\infty)$
such that $\Phi(0)=0$ and $\lim_{t\rightarrow\infty}\Phi(t)=\infty$.
The localized Luxembourg norm of a function $f$ with respect to a
Young function $\Phi$ can be defined as follows
\[
\|f\|_{\Phi,Q}=\|f\|_{\Phi(L),Q}=\inf\left\{ \lambda>0\,:\,\frac{1}{|Q|}\int_{Q}\Phi\left(\frac{|f(x)|}{\lambda}\right)dx\leq1\right\}
\]
We note that the case $\Phi(t)=t$ corresponds to the usual average.
We can see these localized norms as a ``different'' way of taking
averages. We can also define the maximal function associated to $\Phi$
as
\[
M_{\Phi}f(x)=\sup_{x\in Q}\|f\|_{\Phi,Q}.
\]
For each Young function there exists an associated complementary Young
function $\bar{\Phi}$ that satisfies the following inequalities
\[
t\leq\Phi^{-1}(t)\bar{\Phi}^{-1}(t)\leq2t.
\]
A basic example of this is $\Phi(t)=t^{p}$. For that Young function
$\bar{\Phi}(t)=t^{p'}$. Some basic facts are that
\[
\frac{1}{|Q|}\int_{Q}|fg|\leq2\|f\|_{\Phi,Q}\|g\|_{\bar{\Phi},Q}
\]
and that
\[
\|M_{\Phi}\|_{L^{p}(\mathbb{R}^{n})}\lesssim\alpha_{p}(\Phi)
\]
where $\alpha_{p}(\Phi)=\left(\int_{1}^{\infty}\frac{\Phi(t)}{t^{p}}\frac{dt}{t}\right)^\frac{1}{p}$ (cf. \cite{Pe1}).

\subsection{$A_{p}$ and bump type conditions}

In this section we give precise definitions of all the ``$A_p$ type'' conditions that appear in this paper. All them resemble in some way their classical counterparts as $A_{p}$ or bump type conditions. We begin with the constant involved in the
weak type $(p,p)$ inequality.
\begin{defn}
Let $1<p<\infty$. Given a weight $\sigma$ and a Borel measure $\mu$
on $\mathbb{R}_{+}^{n+1}$ we define
\[
[\mu,\sigma]_{A'_{p}}:=\sup_{Q}\frac{\mu(\tilde{Q})}{|Q|}\left(\frac{1}{|Q|}\int_{Q}\sigma(x)dx\right)^{p-1}<\infty
\]

\end{defn}
We define now the constants involved in the characterization of the $L^p$ boundedness.
\begin{defn}Let $1<p<\infty$. Given a weight $\sigma \geq 0$ and a Borel measure $\mu$ on $\mathbb{R}^{n+1}$ we define
\[[\mu,\sigma]_{S'_{p}}=\sup_{Q}\left(\frac{\int_{\tilde{Q}}\mathcal{M}(\sigma\chi_{Q})(x,t)^{p}d\mu(x,t)}{\int_{Q}\sigma(x)dx}\right)^{\frac{1}{p}}\]
where the supremum is taken over all cubes of $\mathbb{R}^n$
\end{defn}
The definition of the dyadic variant of this constant is almost the same
\begin{defn}Let $\mathcal{D}$ a dyadic grid. Let $1<p<\infty$. Given a weight $\sigma\geq 0$ and a Borel measure $\mu$ on $\mathbb{R}^{n+1}$ we define
\[[\mu,\sigma]_{S'_{p},\mathcal{D}}=\sup_{Q\in\mathcal{D}}\left(\frac{\int_{\tilde{Q}}\mathcal{M}(\sigma\chi_{Q})(x,t)^{p}d\mu(x,t)}{\int_{Q}\sigma(x)dx}\right)^{\frac{1}{p}}\]
\end{defn}

We observe that the constants we have just defined are quite natural. Indeed, if we consider $\mu(x,t)=w(x)\delta(t)$ where $w$ is a weight and $\delta$ is Dirac's delta we recover the quantitative result obtained in terms of the $S_p$ constant by Moen in \cite{Mo}.

To end this section we give precise definitions of the constants involved in the quantitative sufficient conditions provided in our main results. First we focus on the constants involved in Theorem \ref{Thm:QuantPR}.
\begin{defn}
Given a weight $\sigma$, a Borel measure $\mu$ on $\mathbb{R}_{+}^{n+1}$
and a Young function $\Phi$ we define the quantity
\[
[\mu,\sigma,\Phi]_{A'_{p}}=\sup_{Q}\frac{\mu(\tilde{Q})}{|Q|}\left\Vert \sigma^{\frac{1}{p'}}\right\Vert _{\Phi,Q}^{p}.
\]
We say that $\mu,\sigma$ belong to the $A'_{p,\Phi}$ class if $[\mu,\sigma,\Phi]_{A_{p}}<\infty$.
\end{defn}
We observe that if we choose $\Phi(t)=t^{p'}$ we recover the $[\mu,w]_{A'_{p}}$
constant.

\begin{defn}
Given a weight $\sigma$, a Borel measure $\mu$ on $\mathbb{R}_{+}^{n+1}$
and a Young function $\Phi$ we define the quantity
\[
[\sigma,\Phi]_{W{}_{p}}=\sup_{Q}\frac{1}{\sigma(Q)}\int_{Q}M_{\Phi}\left(\sigma^{\frac{1}{p}}\chi_{Q}\right)^{p}dx
\]

\end{defn}
If we choose $\Phi_{p}(t)=t^{p}$ we obtain
\[
[\sigma,\Phi]_{W\text{}_{p}}=\sup_{Q}\frac{1}{\sigma(Q)}\int_{Q}M\left(\sigma\chi_{Q}\right)dx
\]
which is the well known $A_\infty$ constant that was discovered by Fujii in \cite{F}, rediscovered by Wilson in \cite{W} and shown to be the most suitable one in \cite{HP}  (see also \cite{HPR}).

Let us now turn our attention to the constant involved in Theorem \ref{Thm:QuantLS}.
\begin{defn}
Let $\varepsilon$ be a monotonic increasing function that satisfies 
\[\int_{\frac{1}{2}}^{\infty}\frac{1}{\varepsilon(t)}\frac{dt}{t}=1.
\]
We define the quantity
\[
\left\lceil \mu,\sigma\right\rceil _{p,\varepsilon}=\sup_{Q}\rho_{\sigma,\varepsilon}(Q)\left(\frac{\sigma(Q)}{|Q|}\right)^{p-1}\frac{\mu(\tilde{Q})}{|Q|}
\]
where $\rho_{\sigma,\varepsilon}(Q)=\rho(Q)\varepsilon(\rho(Q))$
and $\rho(Q)=\frac{\int_{Q}M(\sigma\chi_{Q})dx}{\sigma(Q)}$,
\end{defn}

It's worth compairing the constants involved in Theorems \ref{Thm:QuantLS} and \ref{Thm:QuantPR} when we choose $\Phi(t)=t^{p'}$. In that case the constant in Theorem \ref{Thm:QuantPR} is $$\left([\mu,\sigma]_{A'_{p}}[\sigma]_{A_\infty}\right)^{\frac{1}{p}}.$$ 
In contrast with that couple of supremums, the definition of $\left\lceil \mu,\sigma\right\rceil _{p,\varepsilon}^\frac{1}{p}$ ``includes in some way'' the $A'_p$ constant and a ``bumped''  $A_\infty$ constant. It's not clear which of those quantities is larger. 

Observe that since $\rho(Q)\geq 1$ 
$$\left\lceil \mu,\sigma\right\rceil_{p,\varepsilon} \geq \varepsilon(1) 
\sup_{Q} \left(\frac{\sigma(Q)}{|Q|}\right)^{p-1}\frac{\mu(\tilde{Q})}{|Q|}= \varepsilon(1)\, [\mu,\sigma]_{A'_{p}}
$$
which is a not unexpected fact since the strong-type $(p,p)$ implies the weak-type $(p,p)$. Trivially, we have also that
\[[\mu,\sigma]_{A'_{p}}[\sigma]_{A_\infty}
\geq [\mu,\sigma]_{A'_{p}}.\]

\section{Proofs of the main results}

\subsection{Proof of Theorem \ref{Th:Debilpp}}
We first prove that 
\[
[\mu,\sigma]_{A'_p}\lesssim \|\mathcal{M(\cdot\sigma)}\|_{L^p(\mathbb{R}^n,\sigma)\rightarrow L^{p,\infty}(\mathbb{R}_+^{n+1},\mu)}
\]
holds for any $1< p <\infty$. 
Let us assume that $\|\mathcal{M(\cdot\sigma)}\|_{L^p(\mathbb{R}^n,\sigma)\rightarrow L^{p,\infty}(\mathbb{R}_+^{n+1},\mu)}<\infty$ since otherwise there's nothing to prove. We fix a cube $Q$  in  $\mathbb{R}^{n}$. For all $(x,t)\in\tilde{Q}$ by the definition of $\mathcal{M}$
it's clear that
\[
\frac{1}{|Q|}\int_{Q}|f(x)|\sigma(x)dx\leq\mathcal{M}(f\sigma)(x,t).
\]
This yields
\[
\tilde{Q}\subseteq\left\{ (x,t)\in\mathbb{R}_{\text{+ }}^{n+1}\,:\,\mathcal{M}(f\sigma)(x,t)>\frac{1}{|Q|}\int_{Q}|f(x)|\sigma(x)dx\right\}
\]
Then
\[
\mu\left(\tilde{Q}\right)\leq\mu\left\{ (x,t)\in\mathbb{R}_{\text{+ }}^{n+1}\,:\,\mathcal{M}(f\sigma)(x,t)>\frac{1}{|Q|}\int_{Q}|f(x)|\sigma(x)dx\right\}
\]
and recalling that $\|\mathcal{M}(\cdot\sigma)\|_{L^p(\mathbb{R}^n,\sigma)\rightarrow L^{p,\infty}(\mathbb{R}_+^{n+1},\mu)}<\infty$ we have that
\[
\mu\left(\tilde{Q}\right)\leq \frac{\|\mathcal{M}(\cdot\sigma)\|_{L^p(\mathbb{R}^n,w)\rightarrow L^{p,\infty}(\mathbb{R}_+^{n+1},\mu)}^{{p}}}{\left(\frac{1}{|Q|}\int_{Q}|f(x)|\sigma(x)dx\right)^{p}}\int_{\mathbb{R}^{n}}|f|^{p}\sigma.
\]
If we choose $f=\chi_{Q}$ we obtain the desired conclusion.

Now we want to prove that \[\|\mathcal{M}(\cdot\sigma)\|_{L^p(\mathbb{R}^n,\sigma)\rightarrow L^{p,\infty}(\mathbb{R}_+^{n+1},\mu)}\lesssim [\mu,\sigma]_{A'_p}^{{\frac{1}{p}}}.\]
We assume that $[\mu,\sigma]_{A'_p}<\infty$ since otherwise there's nothing to prove. 

Let us call
\[
E_{\alpha}=\left\{ (x,t)\in\mathbb{R}_{\text{+ }}^{n+1}\,:\,\mathcal{M}(f\sigma)(x,t)>\alpha\right\}
\]

We may assume that $f$ is bounded with compact support. Fix $\alpha>0$. If $(x,t)$ then
we can find a cube $R$ containing $x$ with $l(R)\geq t$ and such
that
\[
\frac{1}{|R|}\int_{R}|f(y)\sigma(y)|dy>\alpha.
\]
Let $k$ be the only integer such that
\[
\frac{1}{2^{n(k+1)}}<|R|\leq\frac{1}{2^{kn}}
\]
there is some dyadic cube $Q$ with sidelength $2^{k}$ such that
$\overset{o}{R}\cap Q\not=\emptyset$ and
\[
\int_{R\cap Q}|\sigma(y)f(y)|dy>\frac{\alpha|R|}{2^{n}}>\frac{\alpha|Q|}{4^{n}}
\]
then
\[
\frac{1}{|Q|}\int_{Q}|\sigma(y)f(y)|>\frac{\alpha}{4^{n}}
\]
and we have that
\[
Q\subseteq Q_{j}\in C_{\frac{\alpha}{4^{n}}}
\]
for some $j$ where $C_{\frac{\alpha}{4^{n}}}$ denotes the family
of maximal dyadic cubes $P$ such that
\[
\frac{1}{|P|}\int_{P}|\sigma(y)f(y)|>\frac{\alpha}{4^{n}}
\]
(observe that we can consider such a family since $f$ is integrable).
Also $x\in R\subseteq3Q\subseteq3Q_{j}$. This yields $t\leq l(R)\leq l(3Q_{j})$
and then $(x,t)\in\tilde{3Q_{j}}$. Thus we have seen that
\[
E_{\alpha}\subseteq\bigcup_{j}\tilde{3Q_{j}}.
\]
Then
\[
\begin{split}\mu(E_{\alpha}) & \leq\sum_{j}\mu\left(\tilde{3Q_{j}}\right)=3^{np}\sum_{j}\frac{\mu\left(\tilde{3Q_{j}}\right)}{|3Q_{j}|^{p}}|Q_{j}|^{p}\\
 & \leq\frac{3^{np}}{\alpha^{p}}\sum_{j}\frac{\mu\left(\tilde{3Q_{j}}\right)}{|3Q_{j}|^{p}}\left(\int_{Q_{j}}|\sigma(x)f(x)|dx\right)^{p}\\
 & \leq\frac{3^{np}}{\alpha^{p}}\sum_{j}\frac{\mu\left(\tilde{3Q_{j}}\right)}{|3Q_{j}|^{p}}\left(\int_{Q_{j}}\sigma(x)dx\right)^{p-1}\int_{Q_{j}}|f(x)|^{p}\sigma(x)dx\\
 & \leq\frac{3^{np}}{\alpha^{p}}\sum_{j}\frac{\mu\left(\tilde{3Q_{j}}\right)}{|3Q_{j}|^{p}}\left(\int_{Q_{j}}\sigma(x)dx\right)^{p-1}\int_{Q_{j}}|f(x)|^{p}\sigma(x)dx\\
 & \leq\frac{3^{np}}{\alpha^{p}}[\mu,\sigma]_{A'_{p}}\sum_{j}\int_{Q_{j}}|f(x)|^{p}\sigma(x)dx\\
 & \leq\frac{3^{np}}{\alpha^{p}}[\mu,\sigma]_{A'_{p}}\int_{\mathbb{R}^{n}}|f(x)|^{p}\sigma(x)dx.
\end{split}
\]

\subsection{Proof of Theorem \ref{Thm:CarlesonMax}}

The proof of Theorem \ref{Thm:CarlesonMax} is based on the corresponding dyadic version, namely, 
\begin{thm}
\label{Thm:CarlesonMaxDyad}
Let $\mu$ be a positive Borel measure  on $\mathbb{R}_{+}^{n+1}$ and let $\sigma$ be a weight in $\mathbb{R}^{n}$. Let  $\mathcal{D}$ a dyadic grid and let $1<p<\infty$. Then
\[
\|\mathcal{N^\mathcal{D}}(\sigma\cdot)\|_{L^{p}(\mathbb{R}^{n},\sigma)\rightarrow L^{p}(\mathbb{R}_{+}^{n+1},\mu)} \simeq  [\mu,\sigma]_{S'_{p}}^{\mathcal{D}}
\]
\end{thm}

The proof of Theorem \ref{Thm:CarlesonMaxDyad} consists essentially in tracking the constants in the proof of this result in \cite{RT}.

\begin{proof}[Proof of Theorem \ref{Thm:CarlesonMaxDyad}] The proof is the same for every choice of dyadic grid $\mathcal{D}$ so for simplicity we denote $\mathcal{N}^\mathcal{D}$ by $\mathcal{N}$ and $[\mu,\sigma]_{S'_{p},\mathcal{D}}$ by $[\mu,\sigma]_{S'_{p}}$.

The fact that $[\mu,\sigma]_{S'_{p}}\lesssim\|\mathcal{N}(\sigma\cdot)\|_{L^{p}(\mathbb{R}^{n},\sigma)\rightarrow L^{p}(\mathbb{R}_{+}^{n+1},\mu)}$ is straightforward. It suffices to choose $f=\chi_Q$.
We focus now on the converse. As always we may assume that $f$ is bounded with compact support.
We need to introduce the following truncations of the dyadic maximal operator:
\[
\mathcal{N}^{R}(\sigma f)(x,t)=\sup_{x\in Q,t\leq l(Q)\leq R}\frac{1}{|Q|}\int_{Q}|f(x)|dx
\]
We observe that $\mathcal{N}^{R}f(x,t)=0$ for $t>R$ and that
\[
\lim_{R\rightarrow\infty}\mathcal{N}^{R}(\sigma f)(x,t)\uparrow\mathcal{N}(\sigma f)(x,t).
\]
Let
\[
\Omega_{k}=\left\{ \left(x,t\right)\in\mathbb{R}^{n+1}_+\,:\,\mathcal{N}^{R}(\sigma f)(x,t)>2^{k}\right\} \qquad k\in\mathbb{Z}.
\]
We will also need a Calder\'on-Zygmund decomposition suited to that truncations. The following Lemma contains that decomposition.

\begin{lem}[\cite{RT}] \label{dyadicCZcovering}Let $g\in L^1$. For each $k\in\mathbb{Z}$ we can choose a family $\{Q_{j}^{k}\}_{j\in J_{k}}\subseteq\mathcal{D}$ such that
\begin{enumerate}
\item $\frac{1}{|Q_{j}^{k}|}\int_{Q_{j}^{k}}|g(x)|dx>2^{k}$.
\item The interiors of $\tilde{Q_{j}^{k}}$ are disjoint.
\item $\Omega_{k}=\left\{ \left(x,t\right)\in\mathbb{R}^{n+1}_+\,:\,\mathcal{N}^{R}(g)(x,t)>2^{k}\right\}=\bigcup_{j\in J_{k}}\tilde{Q_{j}^{k}}$
\end{enumerate}
\end{lem}

The proof of that lemma is the same as the one contained \cite{RT} with minor modifications.

Applying Lemma with $g=\sigma f$, we can consider the sets
\[
E_{j}^{k}=\tilde{Q_{j}^{k}}\setminus\left\{ \left(x,t\right)\,:\,\mathcal{N}^{R}(\sigma f)(x,t)\geq2^{k+1}\right\}.
\]
The sets $E_{j}^{k}$ have disjoint interiors and
\[
\begin{split} & \int_{\mathbb{R}_{+}^{n+1}}\mathcal{N}^{R}(\sigma f)(x,t)^{p}d\mu(x,t)\leq\sum_{j,k}\int_{E_{j}^{k}}\mathcal{N}^{R}(\sigma f)(x,t)^{p}d\mu(x,t)\\
&\leq\sum_{j,k}2^{(k+1)p}\mu(E_{j}^{k})\leq2^{p}\sum_{j,k}\left(\frac{1}{|Q_{j}^{k}|}\int_{Q_{j}^{k}}|\sigma(x)f(x)|dx\right)^{p}\mu(E_{j}^{k})
\end{split}
\]
Following ideas from Sawyer and Jawerth we introduce the following notation
\[
\begin{split}\gamma_{jk}=\mu(E_{j}^{k})\left(\frac{\sigma(Q_{j}^{k})}{|Q_{j}^{k}|}\right)^{p} & \qquad g_{jk}=\left(\frac{1}{\sigma(Q_{j}^{k})}\int_{Q_{j}^{k}}|\sigma(x)f(x)|dx\right)^{p}\\
X=\left\{ (k,j)\,:\,k\in\mathbb{Z},\,j\in J_{k}\right\}  & \text{ with atomic measure }\gamma_{jk}\\
\Gamma(\lambda)=\left\{ (k,j)\in X\,:\,g_{jk}>\lambda\right\}
\end{split}
\]
Using this notation we have that
\[
\begin{split} & \int_{\mathbb{R}_{+}^{n+1}}\mathcal{N}^{R}(\sigma f)(x,t)^{p}d\mu(x,t)\leq\sum_{j,k}\gamma_{jk}g_{jk}=\\
 & =2^{p}\int_{0}^{\infty}\gamma\left(\left\{ (k,j)\,:\,g_{jk}>\lambda\right\} \right)d\lambda=2^{p}\int_{0}^{\infty}\left(\sum_{(k,j)\in\Gamma(\lambda)}\gamma_{jk}\right)d\lambda\\
 & =2^{p}\int_{0}^{\infty}\left(\sum_{(k,j)\in\Gamma(\lambda)}\mu(E_{j}^{k})\left(\frac{\sigma(Q_{j}^{k})}{|Q_{j}^{k}|}\right)^{p}\right)d\lambda\\
 & =2^{p}\int_{0}^{\infty}\left(\sum_{(k,j)\in\Gamma(\lambda)}\int_{E_{j}^{k}}\left(\frac{\sigma(Q_{j}^{k})}{|Q_{j}^{k}|}\right)^{p}d\mu(x,t)\right)d\lambda
\end{split}
\]
If we call $Q_{i}$ the maximal cubes of the family $\{Q_{j}^{k}\,:\,(k,j)\in\Gamma(\lambda)\}$
we can rewrite the preceding sum as follows
\[
\begin{split} & 2^{p}\int_{0}^{\infty}\left(\sum_{i}\sum_{\stackrel{{\scriptstyle (k,j)\in\Gamma(\lambda)}}{Q_{j}^{k}\subset Q_{i}}}\int_{E_{j}^{k}}\left(\frac{\sigma(Q_{j}^{k})}{|Q_{j}^{k}|}\right)^{p}d\mu(x,t)\right)d\lambda\\
 & \leq2^{p}\int_{0}^{\infty}\left(\sum_{i}\sum_{\stackrel{{\scriptstyle (k,j)\in\Gamma(\lambda)}}{Q_{j}^{k}\subset Q_{i}}}\int_{E_{j}^{k}}\mathcal{N}^{R}(\sigma\chi_{Q_{i}})(x,t)^{p}d\mu(x,t)\right)d\lambda\\
 & =2^{p}\int_{0}^{\infty}\left(\sum_{i}\int_{\tilde{Q_{i}}}\mathcal{N}^{R}(\sigma\chi_{Q_{i}})(x,t)^{p}d\mu(x,t)\right)d\lambda
\end{split}
\]
Now we use 2 and we obtain
\[
\begin{split} & \leq[\mu,v]_{S'_{p}}^{p}2^{p}\int_{0}^{\infty}\left(\sum_{i}\int_{Q_{i}}\sigma(x)dx\right)d\lambda=2^{p}\int_{0}^{\infty}\sigma\left(\bigcup_{i}Q_{i}\right)d\lambda\\
 & =[\mu,v]_{S'_{p}}^{p}2^{p}\int_{0}^{\infty}\sigma\left(\bigcup_{(k,j)\in\Gamma(\lambda)}Q_{j}^{k}\right)d\lambda
\end{split}
\]
Now from the definition of $\Gamma(\lambda)$ and the $g_{jk}'s$
it's clear that
\[
\bigcup_{(k,j)\in\Gamma(\lambda)}Q_{j}^{k}\subseteq\left\{ x\in\mathbb{R}^{n}\,:\,N_{\sigma}\left(f\right)(x)^{p}>\lambda\right\}
\]
Then
\[
\begin{split} & \leq[\mu,v]_{S'_{p}}^{p}2^{p}\int_{0}^{\infty}\sigma\left(\bigcup_{(k,j)\in\Gamma(\lambda)}Q_{j}^{k}\right)d\lambda\\
 & \leq[\mu,v]_{S'_{p}}^{p}2^{p}\int_{0}^{\infty}\sigma\left(\left\{ x\in\mathbb{R}^{n}\,:\,N_{\sigma}\left(f\right)(x)^{p}>\lambda\right\} \right)d\lambda\\
 & =[\mu,v]_{S'_{p}}^{p}2^{p}\int_{\mathbb{R}^{n}}N_{\sigma}\left(f\right)(x)^{p}\sigma(x)dx\\
 & \lesssim_{n,p}[\mu,v]_{S'_{p}}^{p}\int_{\mathbb{R}^{n}}|f(x)|^{p}\sigma(x)dx
\end{split}
\]
We have obtained that for every $R>0$
\[
\int_{\mathbb{R}_{+}^{n+1}}\mathcal{N}^{R}(\sigma f)(x,t)^{p}d\mu(x,t)\leq c_{n,p}[\mu,v]_{S'_{p}}^{p}\int_{\mathbb{R}^{n}}|f(x)|^{p}\sigma(x)dx.
\]

Since $\mathcal{N}^{R}f(x,t)\uparrow\mathcal{N}f(x,t)$ by monotone
convergence theorem we're done.

\end{proof}

Theorem \ref{Thm:CarlesonMax} will be proved reducing it to the dyadic case. To do that we will
show next that $\mathcal{M}$ is controlled pointwise by a sum of $2^{n}$ $\mathcal{N}^{\mathcal{D}_{j}}$ operators. 

\begin{lem}
\label{Lmm:DyadicControl}We can build $2^{n}$ dyadic grids such
that
\[
\mathcal{M}f(x,t)\leq6^{n}\sum_{j=1}^{2^{n}}\mathcal{N}^{\mathcal{D}_{j}}f(x,t).
\]
\end{lem}
\begin{proof}
To prove this lemma we need to use the fact that we can find a finite
family of dyadic grids such that every cube is ``well controlled''
by a cube in one these families
\begin{lem*}[{\cite[Prop. 5.1]{Le1}}]
There are $2^{n}$ dyadic grids $\mathcal{D}_{j}$ such that for
any cube $Q$, there exists a cube $Q_{j}\in\mathcal{D}_{j}$ such
that $Q\subseteq Q_{j}$ and $l(Q_{j})\leq6l(Q)$.
\end{lem*}
We note that this result appears implicitly in  \cite[p. 136-137]{GCRdF}.

Armed with that lemma we can establish our result. Let us take $(x,t)\in\mathbb{R}_{+}^{n+1}$
and consider
\[
\mathcal{M}f(x,t)=\sup_{x\in Q,\,l(Q)\geq t}\frac{1}{|Q|}\int_{Q}|f(x)|dx.
\]
If we choose any of the averages
\[
\frac{1}{|Q|}\int_{Q}|f(x)|dx
\]
involved in that supremum we can take $Q_{j}\in\mathcal{D}_{j}$ such
that $Q\subseteq Q_{j}$ and $l(Q)\leq l(Q_{j})\leq6l(Q).$ This yields
\[
\frac{1}{|Q|}\int_{Q}|f(x)|dx\leq6^{n}\frac{1}{|Q_{j}|}\int_{Q_{j}}|f(x)|dx\leq6^{n}\mathcal{N}^{\mathcal{D}_{j}}f(x,t)
\]
since $l(Q_{j})\geq l(Q)\geq t$. Consequently
\[
\mathcal{M}f(x,t)\leq6^{n}\sum_{j=1}^{2^{n}}\mathcal{N}^{\mathcal{D}_{j}}f(x,t).
\]

\end{proof}
To end this section we give the proof of Theorem \ref{Thm:CarlesonMax}.
\begin{proof}[Proof of Theorem \ref{Thm:CarlesonMax}]
The proof of $[\mu,\sigma]_{S'_{p}}\lesssim\|\mathcal{M}(\sigma\cdot)\|_{L^{p}(\mathbb{R}^{n},\sigma)\rightarrow L^{p}(\mathbb{R}_{+}^{n+1},\mu)}$ is exactly the same of the dyadic version Theorem \ref{Thm:CarlesonMaxDyad}. It's simply testing with cubes.  

Now suppose that $[\mu,\sigma ]_{S'_{p},\mathcal{D}_{j}}<\infty$, since otherwise there's nothing to prove.  It's clear that for every dyadic grid $\mathcal{D}_j$  
\begin{equation}\label{SpSpD}
[\mu,\sigma]_{S'_{p},\mathcal{D}_{j}}\leq [\mu,\sigma]_{S'_{p}}.
\end{equation}
Consequently Theorem \ref{Thm:CarlesonMaxDyad}
applies for every dyadic grid, that is,
\[
\left(\int_{\mathbb{R}_{+}^{n+1}}\mathcal{N}^{\mathcal{D}_{j}}f(x,t)^{p}d\mu(x,t)\right)^{\frac{1}{p}}\leq c_{n,p} [\mu,\sigma]_{S'_{p},\mathcal{D}_{j}}\left(\int_{\mathbb{R}^{n}}|f(x)|^{p}v(x)dx\right)^{\frac{1}{p}}.
\]
Now we apply Lemma \ref{Lmm:DyadicControl} and use \eqref{SpSpD} to obtain
\[
\begin{split}\left(\int_{\mathbb{R}_{+}^{n+1}}\mathcal{M}f\sigma(x,t)^{p}d\mu(x,t)\right)^{\frac{1}{p}} & \leq\left(\int_{\mathbb{R}_{+}^{n+1}}\left[\sum_{j=1}^{2^{n}}\mathcal{N}^{\mathcal{D}_{j}}(f\sigma)(x,t)\right]^{p}d\mu(x,t)\right)^{\frac{1}{p}}\\
 & \leq\sum_{j=1}^{2^{n}}\left(\int_{\mathbb{R}_{+}^{n+1}}\mathcal{N}^{\mathcal{D}_{j}}(f\sigma)(x,t)^{p}d\mu(x,t)\right)^{\frac{1}{p}}\\
 & \leq\sum_{j=1}^{2^{n}}c_{n,p}[\mu,\sigma]_{S_{p},\mathcal{D}_{j}}\left(\int_{\mathbb{R}^{n}}|f(x)|^{p}\sigma(x)dx\right)^{\frac{1}{p}}\\
 & \leq c_{n,p}[\mu,\sigma]_{S_{p}}\left(\int_{\mathbb{R}^{n}}|f(x)|^{p}\sigma(x)dx\right)^{\frac{1}{p}}
\end{split}
\]
as we wanted to prove.

\end{proof}

\subsection{Proof of Theorems \ref{Thm:QuantPR} and \ref{Thm:QuantLS}}

To proof both Theorems we need a Calderón-Zygmund type decomposition suited to our purposes. We obtain that decomposition in the following

\begin{lem}\label{Lemma:CZ} Let $\mathcal{D}$ be a dyadic grid. Let $f$ be a locally integrable function such that $\supp f \subseteq Q_0\in\mathcal{D}$. Let $k_0\in\mathbb{Z}$ be the smallest integer such that 
$$
\frac{1}{|Q_0|}\int_{Q_0}|f|\leq 2^{k_0}
$$
For each $k\in\mathbb{Z}$ such that $k\geq 0$ we can choose a family $\{Q_{j}^{k}\}_{j\in J_{k}}\subseteq\mathcal{D}(Q_0)$ such that 
\begin{enumerate}
\item $2^{k(n+1)+k_0}<\frac{1}{|Q_{j}^{k}|}\int_{Q_{j}^{k}}|f(x)|dx\leq 2^{k(n+1)+k_0+n}$.
\item The interiors of $\tilde{Q_{j}^{k}}$ with $j\in J_k$ are pairwise disjoint.
\item $\Omega_{k}=\left\{(x,t)\in \tilde{Q_0} : \mathcal{N}^\mathcal{D}f(x,t)>2^{k(n+1)+k_0}\right\}=\bigcup_{j\in J_{k}}\tilde{Q_{j}^{k}}$
\end{enumerate}
Furthermore, the family of cubes $\{Q_{j}^{k}\}_{j\in \cup J_k}$ is sparse, that is, for each $Q_{j}^{k}$ we can take a measurable
subset $E_{j}^{k}$ such that
\[
\frac{1}{2}|Q_{j}^{k}|\leq|E_{j}^{k}|
\]
and the sets $E_{j}^{k}$ are pairwise disjoint.\end{lem}
\begin{proof}
First of all a dyadic version of Lemma \ref{Lemma:CUMP5.38} allows us to say that for each $(x,t)\in \tilde{Q_0}$ we have that
$$
\mathcal{N}^\mathcal{D}f(x,t)=\sup_{x\in Q\in\mathcal D, Q\subseteq Q_0, l(Q)\geq t}\frac{1}{|Q|}\int_Q |f(y)|dy
$$
Then we observe that for each $(x,t)\in \Omega_k$ we can find a cube $Q_{(x,t)}\in\mathcal{D}$ such that $x\in Q_{(x,t)}$ $Q_{(x,t)}\subseteq Q_0$ and $l(Q_{(x,t)})\geq t$. Since all that cubes are contained in $Q_0$ we can choose among them the maximal ones and call them $\{Q_j^k\}_{j\in J_k}$. It's clear that this collection of cubes satisfies conditions (1), (2) and (3). 
To end the proof we have to prove that the family $\{Q_{j}^{k}\}_{\cup_{k\geq k_0}J_k}$ is sparse.

Let us call for each $k\geq k_0$
\[
H_{k}=\left\{ x\in Q\,:\,\frac{1}{|P|}\int_{P}|f(y)|dy>2^{k(n+1)+k_0},\quad\text{for some }P\in\mathcal{D}, P\subseteq Q\right\}
\]
We observe that
\[
|Q_{j}^{k}\cap H_{k+1}|=\sum_{m}|Q_{j}^{k}\cap Q_{m}^{k+1}|
\]
Now since $Q_{j}^{k},Q_{m}^{k+1}\in\mathcal{D}$, we have
that $Q_{j}^{k}\cap Q_{m}^{k+1}\not=\emptyset$ implies that either
$Q_{j}^{k}\subseteq Q_{m}^{k+1}$ or $Q_{m}^{k+1}\subseteq Q_{j}^{k}$.
Now we observe that from the definition of $H_{k}$ it follows that
$Q_{m}^{k+1}\subseteq H_{k}$. Consequently $Q_{m}^{k+1}\subseteq Q_{j}^{k}$
by maximality. Taking that into account and using $(1)$ and $(2)$
\[
\begin{split}|Q_{j}^{k}\cap H_{k+1}| & =\sum_{m}|Q_{j}^{k}\cap Q_{m}^{k+1}|=\sum_{Q_{m}^{k+1}\subseteq Q_{j}^{k}}|Q_{m}^{k+1}| \leq\sum_{Q_{m}^{k+1}\subseteq Q_{j}^{k}}\frac{1}{2^{(k+1)(n+1)+k_0}}\int_{Q_{m}^{k+1}}|f(y)|dy\\
 &\leq\frac{1}{2^{(k+1)(n+1)+k_0}}\int_{Q_{j}^{k}}|f(y)|dy\leq\frac{2^{k(n+1)+k_0+n}}{2^{(k+1)(n+1)+k_0}}|Q_{j}^{k}|=\frac{1}{2}|Q_{j}^{k}|
\end{split}
\]
Consequently we have that
\[
|Q_{j}^{k}|=|Q_{j}^{k}\setminus H_{k+1}|+|Q_{j}^{k}\cap H_{k+1}|\le|Q_{j}^{k}\setminus H_{k+1}|+\frac{1}{2}|Q_{j}^{k}|
\]
and we obtain the desired conclusion taking $E_{j}^{k}=Q_{j}^{k}\setminus H_{k+1}.$
\end{proof}

To give the proof of Theorem \ref{Thm:QuantPR} we will adapt the
argument in P\'erez-Rela \cite{PR}. 

\begin{proof}[Proof of Theorem \ref{Thm:QuantPR}]
As in the proof of Theorem \ref{Thm:CarlesonMax} it is enough to consider $\mathcal{N}^{\mathcal{D}_{j} }$, 
\[
\| \mathcal{N}^{\mathcal{D}_{j} } (\sigma \cdot)   \|_{L^{p}(\mathbb{R}^{n},\sigma)\rightarrow L^{p}(\mathbb{R}_{+}^{n+1},\mu)} \simeq 
[\mu,\sigma]_{S'_{p},\mathcal{D}_{j}}    \qquad j=1,\cdots, 2^n
\]
where
\[
[\mu,\sigma]_{S'_{p},\mathcal{D}_{j}} =\sup_{Q \in \mathcal{D}_{j} }\left(\frac{\int_{\tilde{Q}} \mathcal{N}^{\mathcal{D}_{j} }(\sigma \chi_Q) (x,t)^{p}d\mu(x,t)}{\int_{Q}\sigma(x)dx}\right)^{\frac{1}{p}}.
\]
Hence, it is enough to prove that
\[
[\mu,\sigma]_{S'_{p},\mathcal{D}_j}\lesssim\left([\mu,\sigma,\Phi]_{A'_{p}}[\sigma,\bar{\Phi}]_{W^{p}}\right)^{\frac{1}{p}}.
\]
Let us fix a cube $Q$. Consider $k_{0}\in\mathbb{Z}$ such that
\[
2^{k_{0}-1}<\frac{\sigma(Q)}{|Q|}\leq2^{k_{0}}.
\]
Now we call
\[
A=\left\{ (x,t)\in\tilde{Q}\,;\,\mathcal{M}(\sigma\chi_{Q})(x,t)\leq2\frac{\sigma(Q)}{|Q|}\right\} .
\]
If $(x,t)\in\tilde{Q}\setminus A$ then
\[
\mathcal{M}(\sigma\chi_{Q})(x,t)>2\frac{\sigma(Q)}{|Q|}>2^{k_{0}}\geq\frac{\sigma(Q)}{|Q|}.
\]
We call
\[
\Omega_{k}=\left\{ (x,t)\in\tilde{Q}\,;\,\mathcal{M}(\sigma\chi_{Q})(x,t)>2^{k(n+1)+k_0}\right\} .
\]
Now we see that
\[
\begin{split}\int_{\tilde{Q}}\mathcal{M}(\sigma\chi_{Q})(x,t)^{p}d\mu(x,t) & =\int_{A}\mathcal{M}(\sigma\chi_{Q})(x,t)^{p}d\mu(x,t)+\int_{\tilde{Q}\setminus A}\mathcal{M}(\sigma\chi_{Q})(x,t)^{p}d\mu(x,t)\\
 & \leq2^{p}\left(\frac{\sigma(Q)}{|Q|}\right)^{p}\mu(\tilde{Q})+\sum_{k \geq 0}\int_{\Omega_{k}\setminus\Omega_{k+1}}\mathcal{M}(\sigma\chi_{Q})(x,t)^{p}d\mu(x,t)\\
 & =I+II
\end{split}
\]
We first obtain an estimate for $I$. Using generalized H\"older inequality
\[
\begin{split}I & =2^{p}\left(\frac{1}{|Q|}\int_{Q}\sigma^{\frac{1}{p}}(x)\sigma^{\frac{1}{p'}}(x)dx\right)^{p}\mu(\tilde{Q})\leq2^{p}2^{p}\left\Vert \sigma^{\frac{1}{p}}\right\Vert _{\bar{\Phi},Q}^{p}\left\Vert \sigma^{\frac{1}{p'}}\right\Vert _{\Phi,Q}^{p}\mu(\tilde{Q})\\
 & =2^{p}2^{p}\left\Vert \sigma^{\frac{1}{p}}\right\Vert _{\bar{\Phi},Q}^{p}\left\Vert \sigma^{\frac{1}{p'}}\right\Vert _{\Phi,Q}^{p}\frac{\mu(\tilde{Q})}{|Q|}|Q|\leq2^{p}2^{p}[\mu,\sigma,\Phi]_{A_{p}}\left\Vert \sigma^{\frac{1}{p}}\right\Vert _{\bar{\Phi},Q}^{p}|Q|\\
 & \leq2^{p}2^{p}[\mu,\sigma,\Phi]_{A_{p}}\int_{Q}M_{\bar{\Phi}}\left(\sigma^{\frac{1}{p}}\chi_{Q}\right)^{p}(x)dx
\end{split}
\]
We now focus on $II$
\[
\begin{split}
II&\leq\sum_{k\geq 0}2^{((k+1)(n+1)+k_0)p}\mu(\Omega_k)\\
 &= 2^{(n+1)p}\sum_{k\geq 0}2^{(k(n+1)+k_0)p}\mu(\Omega_k) 
\end{split}
\]
Using now Lemma \ref{Lemma:CZ} we
have that
\[
\begin{split}II\leq & 2^{(n+1)p}\sum_{k\geq 0}2^{(k(n+1)+k_0)p}\mu(\Omega_k) \\
& \leq2^{(n+1)p}\sum_{k,i}\left(\frac{1}{|Q_{i}^{k}|}\int_{Q_{i}^{k}}\sigma(x)\chi_{Q}(x)dx\right)^{p}\mu\left(\widetilde{Q_{i}^{k}}\right)\\
 & \leq2^{(n+1)p}\sum_{k,i}\left(\frac{1}{|Q_{i}^{k}|}\int_{Q_{i}^{k}}\sigma^{\frac{1}{p}}(x)\sigma^{\frac{1}{p'}}(x)\chi_{Q}(x)dx\right)^{p}\frac{\mu\left(\widetilde{Q_{i}^{k}}\right)}{|Q_{i}^{k}|}|Q_{i}^{k}|
\end{split}
\]
Now using generalized H\"older inequality
\[
\begin{split} & 2^{(n+1)p}\sum_{k,i}\left(\frac{1}{|Q_{i}^{k}|}\int_{Q_{i}^{k}}\sigma^{\frac{1}{p}}(x)\sigma^{\frac{1}{p'}}(x)\chi_{Q}(x)dx\right)^{p}\frac{\mu\left(\widetilde{Q_{i}^{k}}\right)}{|Q_{i}^{k}|}|Q_{i}^{k}|\\
 & \leq2^{(n+2)p}\sum_{k,i}\left\Vert \sigma^{\frac{1}{p}}\chi_{Q}\right\Vert _{\bar{\Phi},Q_{i}^{k}}^{p}\left\Vert \sigma^{\frac{1}{p'}}\right\Vert _{\Phi,Q_{i}^{k}}^{p}\frac{\mu\left(\widetilde{Q_{i}^{k}}\right)}{|Q_{i}^{k}|}|Q_{i}^{k}|\\
 & \leq 2^{(n+2)p}[\mu,\sigma,\Phi]_{A'_{p}}\sum_{k,i}\left\Vert \sigma^{\frac{1}{p}}\chi_{Q}\right\Vert _{\bar{\Phi},Q_{i}^{k}}^{p}|Q_{i}^{k}|
\end{split}
\]
Since Lemma \ref{Lemma:CZ} allows us choose $E_{i}^{k}\subseteq Q_{i}^{k}$
pairwise disjoint and such that $|Q_{i}^{k}|\leq2|E_{i}^{k}|$, we have that
\[
\begin{split} &  2^{(n+2)p}[\mu,\sigma,\Phi]_{A'_{p}}\sum_{k,i}\left\Vert \sigma^{\frac{1}{p}}\chi_{Q}\right\Vert _{\bar{\Phi},Q_{i}^{k}}^{p}|Q_{i}^{k}|\\
 & \leq2^{(n+2)p+1}[\mu,\sigma,\Phi]_{A_{p}}\sum_{k,i}\left\Vert \sigma^{\frac{1}{p}}\chi_{Q}\right\Vert _{\bar{\Phi},Q_{i}^{k}}^{p}|E_{i}^{k}|\\
 & \leq2^{(n+2)p+1}[\mu,\sigma,\Phi]_{A_{p}}\sum_{k,i}\int_{E_{i}^{k}}M_{\bar{\Phi}}(\sigma^{\frac{1}{p}}\chi_{Q})^{p}(x)dx\\
 & \leq2^{(n+2)p+1}[\mu,\sigma,\Phi]_{A_{p}}\int_{Q}M_{\bar{\Phi}}(\sigma^{\frac{1}{p}}\chi_{Q})^{p}(x)dx.
\end{split}
\]
Finally, combining estimates
\[
\int_{\tilde{Q}}\mathcal{N}^\mathcal{D}(\sigma\chi_{Q})(x,t)^{p}d\mu(x,t)\leq c_{n,p}[\mu,\sigma,\Phi]_{A'_{p}}\int_{Q}M_{\bar{\Phi}}(\sigma^{\frac{1}{p}}\chi_{Q})^{p}(x)dx,
\]
dividing by $\sigma(Q)$ and raising to the power $\frac{1}{p}$
\[
[\mu,\sigma]_{S'_{p},\mathcal{D}}\leq c_{n,p}\left([\mu,\sigma,\Phi]_{A'_{p}}[\sigma,\bar{\Phi}]_{W_{p}}\right)^{\frac{1}{p}}.
\]

\end{proof}

To end this section we give the proof of Theorem \ref{Thm:QuantLS}. We will adapt the proof of \cite{LS}.

\begin{proof}
As in the proof of Theorem \ref{Thm:CarlesonMax} it is enough to consider $\mathcal{N}^{\mathcal{D}_{j} }$, 
\[
\| \mathcal{N}^{\mathcal{D}_{j} } (\sigma \cdot)   \|_{L^{p}(\mathbb{R}^{n},\sigma)\rightarrow L^{p}(\mathbb{R}_{+}^{n+1},\mu)} \simeq 
[\mu,\sigma]_{S_{p},\mathcal{D}_{j}}    \qquad j=1,\cdots, 2^n
\]
where
\[
[\mu,\sigma]_{S_{p},\mathcal{D}_{j}} =\sup_{Q \in \mathcal{D}_{j} }\left(\frac{\int_{\tilde{Q}} \mathcal{N}^{\mathcal{D}_{j} }(\sigma \chi_Q) (x,t)^{p}d\mu(x,t)}{\int_{Q}\sigma(x)dx}\right)^{\frac{1}{p}}.
\]
Hence, it is enough to prove that
\[
[\mu,\sigma]_{S_{p},\mathcal{D}_{j}}    \lesssim  \left\lceil \mu,\sigma\right\rceil_{p,\varepsilon}.
\]
We recall that the quantity $\left\lceil \mu,\sigma\right\rceil_{p,\varepsilon}$ is defined by
\[
\left\lceil \mu,\sigma\right\rceil _{p,\varepsilon}=\sup_{Q}\rho_{\sigma,\varepsilon}(Q)\left(\frac{\sigma(Q)}{|Q|}\right)^{p-1}\frac{\mu(\tilde{Q})}{|Q|}
\]
with $\rho_{\sigma,\varepsilon}(Q)=\rho(Q)\varepsilon(\rho(Q))$
and $\rho(Q)=\frac{\int_{Q}M(\sigma\chi_{Q})dx}{\sigma(Q)}$, where $\varepsilon$ is a monotonic increasing function that satisfies 
\begin{equation}\label{CondEpsilon}
\int_{\frac{1}{2}}^{\infty}\frac{1}{\varepsilon(t)}\frac{dt}{t}=1.
\end{equation}

To simplify the notation we fix one of these $j$ and denote   ${\mathcal{D}_{j}}$ and   $\mathcal{N}^{\mathcal{D}_{j}}$ by $\mathcal{D}$ 
and  $\mathcal{N}$ 

Let us fix $Q\in \mathcal{D}$. Arguing as we did in the proof of Theorem  \ref{Thm:QuantPR}, we can write
\[
\begin{split} & \int_{\tilde{Q}}\mathcal{N}(\sigma\chi_{Q})(x,t)^{p}d\mu(x,t)\\
& \leq 2^{p}\left(\frac{\sigma(Q)}{|Q|}\right)^{p}\mu(\tilde{Q})+2^{(n+1)p}\sum_{k,i}\left(\frac{1}{|Q_{i}^{k}|}\int_{Q_{i}^{k}}\sigma(x)\chi_{Q}(x)dx\right)^{p}\mu\left(\widetilde{Q_{i}^{k}}\right)\\
&=I+II.
\end{split}
\]
First we observe that
\[\begin{split}
I &= 2^{p}\left(\frac{\sigma(Q)}{|Q|}\right)^{p-1}\frac{\mu(\tilde{Q})}{|Q|}\sigma(Q)\leq 2^{p}\left(\frac{\sigma(Q)}{|Q|}\right)^{p-1}\frac{\mu(\tilde{Q})}{|Q|}\rho(Q)\sigma(Q)\\
&\leq \frac{2^p}{\varepsilon(1)}2^{p}\left(\frac{\sigma(Q)}{|Q|}\right)^{p-1}\frac{\mu(\tilde{Q})}{|Q|}\rho_{\sigma,\varepsilon}(Q)\sigma(Q)\leq \frac{2^p}{\varepsilon(1)}\left\lceil \mu,\sigma\right\rceil_{p,\varepsilon}
\end{split}\]

To end the proof we have to control $II$. To simplify we denote by $\mathcal{S}$ the family of cubes $\{Q_{i}^{k}\}$ and everything is left is to understand 
$$II\leq 2^{(n+1)p}\sum_{S\in\mathcal{S}}\sigma_{S}^{p}\mu(\tilde{S})
$$
where $\sigma_{S}= \frac{\sigma(S)}{|S|}$. We divide now the collection $\mathcal{S}$ into subcollections $\mathcal{S}_{a,r}$ as follows.
$S\in\mathcal{S}_{a,r}$ for some $a\in\mathbb{Z}$ and $r\in\{0,1,2,\dots\}$ if and only if
$$
2^{a-1}\leq\sigma_{S}^{p-1}\frac{\mu(\tilde{S})}{|S|}\rho_{\sigma,\varepsilon}(S)\leq2^{a}\quad\text{and}\quad2^{r}\leq\rho(S)\leq2^{r+1}.
$$
Note that $\mathcal{S}_{a,r}$ is empty if $\left\lceil \mu, \sigma\right\rceil _{p,\varepsilon}<2^{a-1}$. Now for these collections we have that
\[
\begin{split}\sum_{S\in\mathcal{S}_{a,r}}\sigma_{S}^{p}\mu(\tilde{S}) & =\sum_{S\in\mathcal{S}_{a,r}}\sigma_{S}^{p-1}\sigma(S)\frac{\mu(\tilde{S})}{|S|}=\sum_{S\in\mathcal{S}_{a,r}}\frac{\sigma_{S}^{p-1}\frac{\mu(\tilde{S})}{|S|}\rho_{\sigma,\varepsilon}(S)\sigma(S)}{\rho_{\sigma,\varepsilon}(S)}\\
 & \leq2^{a}\sum_{S\in\mathcal{S}_{a,r}}\frac{\sigma(S)}{\rho_{\sigma,\varepsilon}(S)}.
\end{split}
\]
Now we observe that $\frac{1}{\rho(S)}\leq\frac{1}{2^{r}}$ and since
$\varepsilon$ is increasing $\varepsilon(2^{r})\leq\varepsilon(\rho(S))$.
Then
\[
\frac{1}{\rho_{\sigma,\varepsilon}(S)}\leq\frac{1}{2^{r}\varepsilon(2^{r})}.
\]
Now we recall that, by Lemma \ref{Lemma:CZ}, $\mathcal{S}$ is a sparse family, that is,  for each $S$ there is a measurable set $E(S) \subset S$ such that $|S| \leq 2|E(S)|$ for each $S$ and such that the family 
$\{E(S)\}_S$ is pairwise disjoint. Then,  
\[
\begin{split}\sum_{S\in\mathcal{S}_{a,r}}\sigma_{S}^{p}\mu(\tilde{S}) & \leq2^{a}\sum_{S\in\mathcal{S}_{a,r}}\frac{\sigma(S)}{\rho_{\sigma,\varepsilon}(S)}\leq2^{a}\sum_{S\in\mathcal{S}_{a,r}}\frac{\sigma(S)}{2^{r}\varepsilon(2^{r})}=\frac{2^{a}}{2^{r}\varepsilon(2^{r})}\sum_{S\in\mathcal{S}_{a,r}}\sigma(S)\\
 & =\frac{2^{a}}{2^{r}\varepsilon(2^{r})}\sum_{\stackrel{{\scriptstyle S'\in\mathcal{S}_{a,r}}}{S'\text{ maximal}}}\sum_{\stackrel{{\scriptstyle S\in\mathcal{S}_{a,r}}}{S\subseteq S'}}\sigma(S)=\frac{2^{a}}{2^{r}\varepsilon(2^{r})}\sum_{\stackrel{{\scriptstyle S'\in\mathcal{S}_{a,r}}}{S'\text{ maximal}}}\sum_{\stackrel{{\scriptstyle S\in\mathcal{S}_{a,r}}}{S\subseteq S'}}\frac{\sigma(S)}{|S|}|S|\\
\mathcal{S}\text{ is sparse} & \leq c\frac{2^{a}}{2^{r}\varepsilon(2^{r})}\sum_{\stackrel{{\scriptstyle S'\in\mathcal{S}_{a,r}}}{S'\text{ maximal}}}\sum_{\stackrel{{\scriptstyle S\in\mathcal{S}_{a,r}}}{S\subseteq S'}}\frac{\sigma(S)}{|S|}|E(S)|\\
 &    \leq c\frac{2^{a}}{2^{r}\varepsilon(2^{r})}\sum_{\stackrel{{\scriptstyle S'\in\mathcal{S}_{a,r}}}{S'\text{ maximal}}}\sum_{\stackrel{{\scriptstyle S\in\mathcal{S}_{a,r}}}{S\subseteq S'}}\left(\inf_{z\in S}M(\sigma\chi_{S'})(z)\right)|E(S)| \\
 & \leq c\frac{2^{a}}{2^{r}\varepsilon(2^{r})}\sum_{\stackrel{{\scriptstyle S'\in\mathcal{S}_{a,r}}}{S'\text{ maximal}}}\sum_{\stackrel{{\scriptstyle S\in\mathcal{S}_{a,r}}}{S\subseteq S'}}\int_{E(S)}M(\sigma\chi_{S'})(x)\,dx\\
 &\leq c\frac{2^{a}}{2^{r}\varepsilon(2^{r})}\sum_{\stackrel{{\scriptstyle S\in\mathcal{S}_{a,r}}}{S\text{ maximal}}}\int_{S}M(\sigma\chi_{S})dx\\
 & \leq2c\frac{2^{a}}{\varepsilon(2^{r})}\sum_{\stackrel{{\scriptstyle S\in\mathcal{S}_{a,r}}}{S\text{ maximal}}}\frac{\int_{S}M(\sigma\chi_{S})dx}{2^{r+1}}\leq2c\frac{2^{a}}{\varepsilon(2^{r})}\sum_{\stackrel{{\scriptstyle S\in\mathcal{S}_{a,r}}}{S\text{ maximal}}}\frac{\int_{S}M(\sigma\chi_{S})dx}{\rho(S)}\\
 & =2c\frac{2^{a}}{\varepsilon(2^{r})}\sum_{\stackrel{{\scriptstyle S\in\mathcal{S}_{a,r}}}{S\text{ maximal}}}\sigma(S)\leq2c\frac{2^{a}}{\varepsilon(2^{r})}\sigma(Q).
\end{split}
\]

Taking this estimate into account we have that
\[
\begin{split}\sum_{S\in\mathcal{S}}\sigma_{S}^{p}\mu(\tilde{S}) & =\sum_{\stackrel{{\scriptstyle a\in\mathbb{Z}}}{\left\lceil \mu,\sigma\right\rceil _{p,\varepsilon}\geq2^{a-1}}}\sum_{r\in\{0,1,\dots\}}\sum_{S\in\mathcal{S}_{a,r}}\sigma_{S}^{p}\mu(\tilde{S})\leq\sum_{\stackrel{{\scriptstyle a\in\mathbb{Z}}}{\left\lceil \mu,\sigma\right\rceil _{p,\varepsilon}\geq2^{a-1}}}\sum_{r\in\{0,1,\dots\}}2c\frac{2^{a}}{\varepsilon(2^{r})}\sigma(Q)\\
 & =2c\left(\sum_{\stackrel{{\scriptstyle a\in\mathbb{Z}}}{\left\lceil \mu,\sigma\right\rceil _{p,\varepsilon}\geq2^{a-1}}}2^{a}\right)\left(\sum_{r\in\{0,1,\dots\}}\frac{1}{\varepsilon(2^{r})}\right)\sigma(Q)\dot{=}2cAB\sigma(Q)
\end{split}
\]
and we are left with controlling both sums.

It's a straightforward computation that
\[A\leq 2\left\lceil \mu,\sigma\right\rceil _{p,\varepsilon}.\]

For $B$, using \eqref{CondEpsilon}, it's clear that
\[\sum_{r\in\{0,1,\dots\}}\frac{1}{\varepsilon(2^{r})}  =\frac{1}{\log(2)}\sum_{r\in\{0,1,\dots\}}\frac{\log(2^{r})-\log(2^{r-1})}{\varepsilon(2^{r})}\approx\int_{\frac{1}{2}}^{\infty}\frac{1}{\varepsilon(t)}\frac{dt}{t}=1. 
\]

Combining these estimates we conclude that
\[\sum_{S\in\mathcal{S}}\sigma_{S}^{p}\mu(\tilde{S}) \leq c\, \left\lceil \mu,\sigma\right\rceil _{p,\varepsilon}\sigma(Q).\]This ends the proof.
\end{proof}

\end{document}